\documentclass[10pt]{amsart}
\usepackage{amsthm}
\usepackage{tikz-cd} 
\usetikzlibrary{cd}
\usepackage{amsmath,amssymb} 
\usepackage{graphicx} 
\usepackage{geometry}
\usepackage[colorlinks,linkcolor=red,anchorcolor=blue,citecolor=green]{hyperref}
\newtheorem{theorem}{Theorem}[section]
\newtheorem{corollary}[theorem]{Corollary}
\newtheorem{lemma}[theorem]{Lemma}
\newtheorem{proposition}[theorem]{Proposition}
\theoremstyle{definition}
\newtheorem{definition}[theorem]{Definition}
\newtheorem{remark}[theorem]{Remark}
\newtheorem{example}[theorem]{Example}

\numberwithin{equation}{section}

\usepackage{indentfirst}
\title{GL-racks and coloring invariants of Legendrian knots}
\author{Zhiyun Cheng}
\address{School of Mathematical Sciences, Laboratory of Mathematics and Complex Systems, MOE, Beijing Normal University, Beijing 100875, China}
\email{czy@bnu.edu.cn}
\author{Zhiyi He}
\address{School of Mathematical Sciences, Beijing Normal University, Beijing 100875, China}
\email{202321130080@mail.bnu.edu.cn}
\begin{document}

\begin{abstract}
In this paper, we explore the algebraic structure of GL-racks, and demonstrate that finite GL-racks decompose canonically into permutation GL-racks and block GL-racks. As a corollary, we verify that two Legendrian knots with the same classical invariants share equivalent coloring invariants with respect to any given finite GL-rack. This partially answers a question proposed in our previous paper \cite{Fundamental-GL-rack-and-classical-invariants}.
\end{abstract}
\maketitle

\section{Introduction}
Quandles and racks are non-associative algebraic structures intimately associated with knots and links, whose axioms are motivated from the Reidemeister moves in knot theory. The quandle concept was independently proposed by Joyce~\cite{Joyce-1982} and Matveev~\cite{Matveev-1982} in 1982. They constructed a canonical algebraic counterpart to the knot group, referred to the fundamental quandle or knot quandle. For a given knot $K$, the knot quandle, denoted by $Q(K)$, serves as a almost complete knot invariant. As a natural generalization of quandles, racks were introduced by Fenn and Rourke~\cite{rack-1992} in 1992 during their investigation of framed knots within 3-manifold topology. A wide range of knot invariants can be constructed via quandle structures, among which the quandle coloring invariant is the most elementary. This invariant counts the number of homomorphisms from the knot quandle $Q(K)$ to a finite quandle. In addition, quandle homology and cohomology theories have been well-established in the literature~\cite{quandle-cohomology}. Specifically, $2$-cocycles and $3$-cocycles give rise to enhanced coloring invariants for classical knots and knotted surfaces respectively. Beyond knot theory, quandle theory also intersects with other research directions, including the study of pointed Hopf algebras~\cite{AG2003}, Yang-Baxter deformations~\cite{Eis2005,Eis2014} and Holonomy braidings~\cite{BGPR2020}. Comprehensive reviews of foundational theories and recent advances in this field are available in~\cite{survey-developments,book}.

In recent years, quandles and racks have been introduced into contact geometry to study Legendrian knots. In 2017, Kulkarni and Prathamesh~\cite{On-Rack-Invariants-2017} first introduced rack invariants of Legendrian knots, which distinguish certain Legendrian unknots. Subsequently, in 2021, Ceniceros, Elhamdadi and Nelson~\cite{2021-Ceniceros-Elhamdadi-Nelson} defined the Legendrian rack, an algebraic structure based on the Legendrian Reidemeister moves for the front projection of a Legendrian knot, together with an automorphism $f$ assigned at each cusp, which is compatible with the underlying rack structure. This allows one to define coloring invariant that distinguish some Legendrian knots from their stabilizations. Generalizing the Legendrian rack, instead of using a single automorphism $f$ at cusps, one can assign automorphisms $u$ and $d$ to the upper and lower cusps of an oriented Legendrian knot, which are required to be compatible with the rack structure. This generalized algebraic structure is called \textit{generalized Legendrian rack} or \textit{bi-Legendrian rack}, simply denoted by GL-rack. This notion was independently introduced by Kimura~\cite{Bi-Legendrian-2023} and Karmakar-Saraf-Singh~\cite{GL-rack-Karmakar-Saraf-Singh}. For Legendrian knots, one can define an invariant analogous to the fundamental quandle, which is called the \textit{fundamental GL-rack}. Recently, it was shown in \cite{Fundamental-GL-rack-and-classical-invariants} that if two Legendrian knots of the same topological knot type have isomorphic fundamental GL-racks, then they either have the same Thurston–Bennequin number and the same rotation number, or opposite Thurston–Bennequin numbers and opposite rotation numbers. A natural question is, are there two Legendrian knots with the same classical invariants but distinct fundamental GL-racks? Or, to what extent the classical invariants of a Legendrian knot can determine its fundamental GL-rack. In this paper, in order to answer these questions we prove the following result.

\begin{theorem}\label{theorem1.1}
If two Legendrian knots $K_1$ and $K_2$ have the same classical invariants, then for any finite GL-rack $(X,\ast,u,d)$, we have $\operatorname{Col}_X(K_1)=\operatorname{Col}_X(K_2)$.
\end{theorem}

If the case of opposite sign in the result mentioned above can be ruled out, then as invariants of Legendrian knots, the Thurston-Bennequin number and rotation number can be inserted between the fundamental GL-rack $\operatorname{GLR}(K)$ and the coloring invariant $\operatorname{Col}_X(K)=|\operatorname{Hom}(\operatorname{GLR}(K), X)|$ for any finite GL-rack $(X, \ast, u, d)$.

The rest of this paper is organized as follows. In Section \ref{section2}, we review background material on quandles, GL-racks, and Legendrian knots. In Section \ref{section3}, we analyze the internal structure of GL-racks, showing that any GL-rack decomposes into permutation GL-racks and block GL-racks. The corresponding coloring invariants of Legendrian knots are discussed. In Section \ref{section4}, by using the results from Section \ref{section3}, we prove that two Legendrian knots with the same classical invariants have identical coloring invariant with respect to any given finite GL-rack. Some additional results are discussed in Section \ref{section5}.

\section{Preliminaries}\label{section2}
In this section, we review some background material related to GL-racks and Legendrian knots.
\subsection{Legendrian knots} A Legendrian knot $K$ in a $3$-manifold equipped with a contact structure is a smoothly embedded $S^1$ whose tangent vectors lie in the contact planes. More often, we consider $\mathbb{R}^3$ with the standard contact structure $\xi_{\text{std}}=\text{span}\{\frac{\partial}{\partial y}, \frac{\partial}{\partial x}+y\frac{\partial}{\partial z}\}$. As in the case of topological knots, we are interested in their classification up to Legendrian isotopy. Legendrian knots admit two distinct projections: the Lagrangian projection and the front projection. In this paper, we consider the front projection $\Pi:\mathbb{R}^3\to\mathbb{R}^2:(x, y, z)\to(x, z)$. Here we present the formulas for computing the Thurston-Bennequin number $tb$ and the rotation number $rot$, two classical invariants of Legendrian knots from a given front projection, as well as their behaviors under stabilization. For more details on Legendrian knots, as well as the geometric meanings and properties of these two classical invariants, we refer the reader to~\cite{Etnyre}. Let $D$ be the front diagram of a given oriented Legendrian knot $K$. We use $u(D)$ and $d(D)$ to denote the number of upward-pointing cusps and the number of downward-pointing cusps of $D$, respectively, and use $w(D)$ to denote the writhe of $D$, that is, the number of positive crossings minus the number of negative crossings. Then the Thurston-Bennequin number and the rotation number are given by the following formulas:
\begin{center}
$tb(K)=w(D)-\frac{1}{2}\bigl(u(D)+d(D)\bigr) \text{ and } rot(K)=\frac{1}{2}\bigl(d(D)-u(D)\bigr)$.
\end{center}

Figure \ref{fig1} illustrates the stabilization operations on a Legendrian knot $K$, denoted by $S_+$ and $S_-$ for positive stabilization and negative stabilization, respectively. The Thurston-Bennequin number and the rotation number transform under stabilization as follows:
\begin{center}
$tb(S_{\pm}(K)) = tb(K) - 1 \text{ and } rot(S_{\pm}(K)) = rot(K) \pm 1$.
\end{center}
Stabilization plays an essential role for Legendrian knots, as can be seen from the following theorem.

\tikzset{every picture/.style={line width=0.75pt}} 
\begin{figure}[h]
\begin{tikzpicture}[x=0.75pt,y=0.75pt,yscale=-1,xscale=1]

\draw    (550.72,85.58) .. controls (556.44,85.21) and (586.58,83.3) .. (596.11,95.47) .. controls (605.65,107.64) and (624.51,109.91) .. (625.78,109.91) .. controls (627.05,109.91) and (617.47,109.37) .. (606.03,112.69) .. controls (594.59,116) and (589.47,122.53) .. (580.92,125.96) .. controls (572.37,129.39) and (553.31,130.93) .. (553.31,130.93) .. controls (553.31,130.93) and (575.16,131.84) .. (593.99,144.21) .. controls (612.83,156.57) and (628.26,152.63) .. (626.14,153.55) ;
\draw   (573.93,82.24) .. controls (576.69,84.86) and (579.65,86.64) .. (582.78,87.61) .. controls (579.46,87.47) and (575.95,88.18) .. (572.25,89.7) ;
\draw    (626.19,180.66) .. controls (620.46,180.24) and (590.27,178.07) .. (580.65,191.19) .. controls (571.04,204.31) and (552.13,206.7) .. (550.86,206.69) .. controls (549.58,206.69) and (559.18,206.13) .. (570.64,209.76) .. controls (582.09,213.39) and (587.18,220.45) .. (595.74,224.2) .. controls (604.3,227.94) and (623.39,229.66) .. (623.39,229.66) .. controls (623.39,229.66) and (602.52,229.04) .. (583.58,242.34) .. controls (564.65,255.64) and (549.1,250.98) .. (551.22,251.98) ;
\draw   (588.14,180.51) .. controls (591.99,181.74) and (595.58,182.14) .. (598.88,181.69) .. controls (595.86,182.98) and (593.12,185.09) .. (590.67,188.06) ;
\draw    (451.63,137.87) -- (507.35,122.67) ;
\draw [shift={(509.28,122.15)}, rotate = 164.75] [color={rgb, 255:red, 0; green, 0; blue, 0 }  ][line width=0.75]    (10.93,-3.29) .. controls (6.95,-1.4) and (3.31,-0.3) .. (0,0) .. controls (3.31,0.3) and (6.95,1.4) .. (10.93,3.29)   ;
\draw    (454.98,197.16) -- (511.82,209.36) ;
\draw [shift={(513.78,209.78)}, rotate = 192.11] [color={rgb, 255:red, 0; green, 0; blue, 0 }  ][line width=0.75]    (10.93,-3.29) .. controls (6.95,-1.4) and (3.31,-0.3) .. (0,0) .. controls (3.31,0.3) and (6.95,1.4) .. (10.93,3.29)   ;
\draw    (306.67,164.61) .. controls (355.95,145.71) and (412.42,156.09) .. (437.72,170.2) ;
\draw   (354.42,150.5) .. controls (358.42,152.57) and (362.42,153.82) .. (366.41,154.23) .. controls (362.42,154.64) and (358.42,155.89) .. (354.42,157.96) ;

\draw (451.7,102.91) node [anchor=north west][inner sep=0.75pt]   [align=left] {$\displaystyle S_{+}$};
\draw (451.44,213.1) node [anchor=north west][inner sep=0.75pt]   [align=left] {$\displaystyle S_{-}$};
\end{tikzpicture}
\caption{Stabilization}
\label{fig1}
\end{figure}
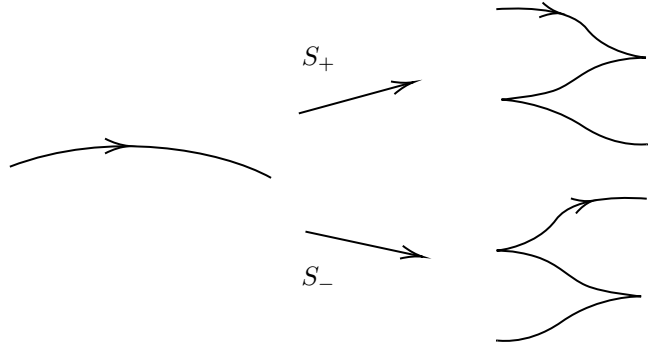

\begin{theorem}\cite{Fuchs-Tabachnikov-1997}\label{theorem2.1}
Any two Legendrian knots of the same knot type become Legendrian isotopic after a sequence of stabilizations.
\end{theorem}

\subsection{Generalized Legendrian racks} A \textit{generalized Legendrian rack}, abbreviated as \textit{GL-rack}, is an algebraic structure generalizing the classical rack to accommodate the study of Legendrian knots. Differently from an ordinary rack, a GL-rack is equipped with two additional compatible operations, denoted by $u$ and $d$. These additional operations enable the construction of invariants for Legendrian knots. In what follows, we provide the formal definition of a GL-rack along with some illustrative examples.

\begin{definition}
A rack is a pair $(X, \ast)$, where $X$ is a set and $\ast$ is a binary operation on $X$ satisfying the following conditions:
\begin{enumerate}
\item For any $y, z\in X$, there exists a unique $x\in X$ such that $x\ast y=z$;
\item $(x\ast y)\ast z=(x\ast z) \ast (y\ast z)$ for all $x,y,z\in X$.
\end{enumerate}
In other words, for any $x\in X$, the operation $\ast x: X\to X$ defines an automorphism of $X$. If the operation $\ast$ further satisfies $x\ast x=x$ for all $x\in X$, then $(X, \ast)$ is called a \textit{quandle} \cite{Joyce-1982,Matveev-1982}. These three axioms correspond to the three Reidemeister moves in knot theory.
\end{definition}

\begin{definition}\cite{GL-rack-Karmakar-Saraf-Singh,Bi-Legendrian-2023}\label{definition2.3}
    A \textit{generalized Legendrian rack} (GL-rack) is a quadruple $(X, \ast, u, d)$, where $(X, \ast)$ is a rack and $u, d$ are two maps on $X$ subject to the following conditions:
\begin{enumerate}
\item[(1)] $ud(x\ast x)=du(x\ast x)=x$;
\item[(2)] $u(x\ast y)=u(x)\ast y$, $d(x\ast y)=d(x)\ast y$;
\item[(3)] $x\ast u(y)=x\ast d(y)=x\ast y$.
\end{enumerate}
\end{definition}

In particular, a GL-rack $(X, \ast, u, d)$ is called a \textit{GL-quandle} if $x\ast x=x$ for all $x\in X$. Note that for a GL-quandle, $u$ and $d$ are inverses of each other, since $ud(x)=du(x)=x$.

Recently, Ta conducted a detailed and thorough study on the algebraic structure of GL-rack \cite{Ta2025}. In particular, Ta showed that the definition of GL-rack given in Definition \ref{definition2.3} can be greatly simplified as follows: a GL-rack is a rack $\{X, \ast\}$ equipped with an automorphism $u$ such that $u(x\ast y)=u(x)\ast y$ for any $x, y\in X$. Here the map $d$ in Definition \ref{definition2.3} can be expressed as $d(x)=u^{-1}(x)\ast^{-1}u^{-1}(x)$. Although the definition given by Ta is more concise than the original one, we still prefer to write down the Definition \ref{definition2.3}, as it relates more directly to the coloring invariants of Legendrian knots. Moreover, the other conditions in Definition \ref{definition2.3} can be regarded as properties that we will frequently use later.

Next, we present an example of GL-rack, say \textit{permutation GL-rack}, which will be frequently used later.

\begin{example}\label{example2.4}
Let $X$ be a finite set and $\sigma$ a permutation on $X$. A special example of GL-rack is the permutation GL-rack $(X,\ast_{\sigma})$, where $x \ast_{\sigma} y = \sigma(x)$ and the maps $u$ and $d$ need to satisfy $ud = \sigma^{-1}$. For example, consider the permutation rack $X = \{1,2,3\}$ with $\sigma=(123)$, $u=\text{id}$, and $d=(132)$. The operation table for $x \ast y$ is given as follows:
 \[
 \begin{array}{c|ccc}
 \ast & 1 & 2 & 3 \\
 \hline
 1 & 2 & 2 & 2 \\
 2 & 3 & 3 & 3 \\
 3 & 1 & 1 & 1 \\
 \end{array}
 \]
which satisfies the defining conditions of a permutation GL-rack.
\end{example}

Similar to the coloring invariant of topological knots with respect to a given finite quandle, we can define the coloring invariant of Legendrian knots with respect to a given finite GL-rack $X$. By a coloring, we mean an assignment of elements of $X$ to every semi-arc in the front projection of a Legendrian knot $K$. Here a semi-arc refers to a segment of the diagram bounded by two undercrossings, two cusps, or one undercrossing and one cusp. The compatibility of the maps $u$ and $d$, together with the rules at crossings, ensures that the number of colorings is an invariant of Legendrian knots. The rules of coloring at cusps and crossings are shown in Figure \ref{fig2}.

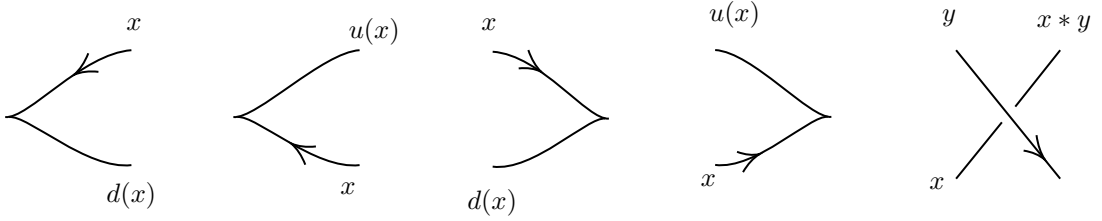
\begin{figure}[h]
    \centering
\tikzset{every picture/.style={line width=0.75pt}} 
\begin{tikzpicture}[x=0.75pt,y=0.75pt,yscale=-1,xscale=1]

\draw    (354.25,109.54) .. controls (378.68,111.17) and (401.8,144.13) .. (412.57,143.08) ;
\draw    (354.25,167.18) .. controls (374.89,169.8) and (402.03,141.93) .. (410.77,143.08) ;
\draw    (586.75,108.77) -- (638.97,173.22) ;
\draw    (638.97,108.77) -- (616.45,137.09) ;
\draw    (609.72,144.95) -- (586.75,173.24) ;
\draw    (465.87,108.69) .. controls (484.76,109.81) and (513.42,143.27) .. (524.18,142.22) ;
\draw    (465.87,166.32) .. controls (486.5,168.94) and (513.65,141.08) .. (522.39,142.22) ;
\draw   (372.15,108.69) .. controls (373.76,113.35) and (375.91,116.99) .. (378.6,119.61) .. controls (375.37,118.02) and (371.61,117.45) .. (367.31,117.91) ;
\draw   (476.85,159.74) .. controls (480.96,161.26) and (484.75,161.64) .. (488.19,160.89) .. controls (485.08,162.77) and (482.3,165.8) .. (479.85,169.95) ;
\draw   (627.09,151.39) .. controls (627.77,156.36) and (629.18,160.46) .. (631.29,163.71) .. controls (628.46,161.33) and (624.9,159.81) .. (620.62,159.15) ;
\draw    (173.02,108.77) .. controls (148.07,110.63) and (121.39,143.35) .. (109.7,142.31) ;
\draw    (173.02,166.41) .. controls (150.62,169.03) and (121.14,141.16) .. (111.65,142.31) ;
\draw   (151.4,110.32) .. controls (149.56,114.96) and (147.16,118.57) .. (144.2,121.13) .. controls (147.73,119.6) and (151.82,119.1) .. (156.48,119.63) ;
\draw    (287.49,108.77) .. controls (268.98,109.87) and (235.86,143.35) .. (224.17,142.31) ;
\draw    (287.49,166.41) .. controls (265.09,169.03) and (235.61,141.16) .. (226.12,142.31) ;
\draw   (265.58,157.05) .. controls (260.84,157.56) and (256.69,157.07) .. (253.13,155.54) .. controls (256.1,158.08) and (258.47,161.64) .. (260.25,166.21) ;

\draw (347.13,91.21) node [anchor=north west][inner sep=0.75pt]   [align=left] {$\displaystyle x$};
\draw (457.19,169.64) node [anchor=north west][inner sep=0.75pt]   [align=left] {$\displaystyle x$};
\draw (571.88,171.29) node [anchor=north west][inner sep=0.75pt]   [align=left] {$\displaystyle x$};
\draw (461.36,82.87) node [anchor=north west][inner sep=0.75pt]   [align=left] {$\displaystyle u( x)$};
\draw (340.07,176.78) node [anchor=north west][inner sep=0.75pt]   [align=left] {$\displaystyle d( x)$};
\draw (578.34,87.03) node [anchor=north west][inner sep=0.75pt]   [align=left] {$\displaystyle y$};
\draw (625.75,89.55) node [anchor=north west][inner sep=0.75pt]   [align=left] {$\displaystyle x\ast y$};
\draw (159.98,176.01) node [anchor=north east][inner sep=0.75pt]  [xscale=-1] [align=left] {$ $};
\draw (169.07,91.21) node [anchor=north west][inner sep=0.75pt]   [align=left] {$\displaystyle x$};
\draw (159.07,173.46) node [anchor=north west][inner sep=0.75pt]   [align=left] {$\displaystyle d( x)$};
\draw (274.45,176.01) node [anchor=north east][inner sep=0.75pt]  [xscale=-1] [align=left] {$ $};
\draw (280.55,91.21) node [anchor=north west][inner sep=0.75pt]   [align=left] {$\displaystyle u( x)$};
\draw (276.38,173.46) node [anchor=north west][inner sep=0.75pt]   [align=left] {$\displaystyle x$};
\end{tikzpicture}
\caption{Rules of coloring}
\label{fig2}
\end{figure}

\begin{example}\label{example2.5}
Consider the Legendrian knot depicted in Figure \ref{fig3} and the GL-rack introduced in Example \ref{example2.4}, we readily derive the necessary condition $ud(x) = x$ for a valid coloring. By the defining property of permutation racks, this is equivalent to $\sigma^{-1}(x) = x$ for all $x$ in the rack. Since the permutation $\sigma = (123)$ has no fixed points, it follows that the number of colorings of the Legendrian unknot depicted in Figure \ref{fig3} by this GL-rack is $0$.

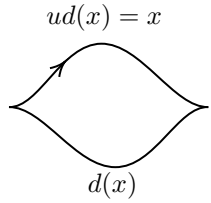
\begin{figure}[h]
\centering
\tikzset{every picture/.style={line width=0.75pt}} 
\begin{tikzpicture}[x=0.75pt,y=0.75pt,yscale=-1,xscale=1]
\draw    (147,116) .. controls (162.25,116.46) and (173.25,84.46) .. (193.25,84.21) .. controls (213.25,83.96) and (229.75,116.46) .. (247,116) ;
\draw    (247,116) .. controls (231.76,115.38) and (220.37,146.17) .. (200.37,146.15) .. controls (180.36,146.12) and (164.25,115.8) .. (147,116) ;
\draw   (167.56,96.48) .. controls (170.58,95.87) and (173.09,94.77) .. (175.06,93.2) .. controls (173.6,95.26) and (172.65,97.83) .. (172.23,100.87) ;

\draw (164.32,62.78) node [anchor=north west][inner sep=0.75pt]  [rotate=-359.94] [align=left] {$\displaystyle ud( x) =x$};
\draw (185,147.11) node [anchor=north west][inner sep=0.75pt]   [align=left] {$\displaystyle d( x)$};
\end{tikzpicture}
\caption{Legendrian unknot}
\label{fig3}
\end{figure}
\end{example}

For topological knots, there exists an almost complete invariant called the \textit{knot quandle}, also known as the \textit{fundamental quandle}~\cite{Joyce-1982,Matveev-1982}. In analogy with this construction, we define the \textit{fundamental GL-rack} for Legendrian knots, which is an invariant of Legendrian knots and induces the corresponding coloring invariant for Legendrian knots with respect to a given finite GL-rack. For a more precise definition, we refer the reader to \cite{GL-rack-Karmakar-Saraf-Singh}.

\begin{definition}
Let $K$ be a Legendrian knot in the standard contact $\mathbb{R}^3$. The \textit{fundamental GL-rack} of $K$, denoted by $\operatorname{GLR}(K)$, is the free GL-rack generated by the arcs of a front projection of $K$, subject to the GL-rack relations assigned at crossings and cusps. 
\end{definition}

For a finite GL-rack $X$, the number of colorings of $K$ with respect to $X$ can be reinterpreted as
 \[
 \operatorname{Col}_X(K) = \bigl\lvert\operatorname{Hom}\bigl(\operatorname{GLR}(K),\, X\bigr)\bigr\rvert,
 \]
 which counts the GL-rack homomorphisms from the fundamental GL-rack of \(K\) to \(X\).

\begin{example}
    The \textit{fundamental GL-rack} of the Legendrian unknot $K$ in Example \ref{example2.5} can be expressed as
    \begin{center}
        $\operatorname{GLR}(K) = \langle x \mid ud(x) = x \rangle.$
    \end{center}
\end{example}

In \cite{GL-rack-Karmakar-Saraf-Singh}, infinitely many distinct Legendrian unknots and Legendrian trefoils are distinguished via the fundamental GL-racks. However, the famous Chekanov-Eliashberg knots have isomorphic fundamental GL-racks.

\section{Structures of GL-racks and coloring invariants of Legendrian knots}\label{section3}
In this section, we analyze the internal structure of GL-racks. We show that any finite GL-rack can be uniquely decomposed into permutation racks and block GL-racks, which act independently of each other. As a corollary, the number of colorings of a Legendrian knot with respect to a given finite GL-rack can be expressed as the sum of the coloring numbers with respect to the corresponding permutation racks and block GL-racks.

\subsection{Structure of GL-racks} We attempt to analyze the ``local properties'' of any given finite GL-rack. Let $(X, \ast, u, d)$ be a finite GL-rack, we define the map $\Delta \colon X \to X$ by
\[
\Delta(x) = x \ast x.
\]
It was proved in \cite[Proposition 1.5]{Szy2018} that $\Delta$ is a rack automorphism. For a GL-rack, it is readily seen that $\Delta = (ud)^{-1}$, and $\Delta$ can be regarded as an element of the symmetric group $S_{|X|}$. Naturally, the same holds for $u$ and $d$. Express $\Delta$ uniquely as a product of disjoint cycles $\Delta = \alpha_1 \cdots \alpha_n$, and let $A_i = \operatorname{supp}(\alpha_i)$ be the support of $\alpha_i$. Then $\Delta|_{A_i} = \alpha_i$,  $A_i\cap A_j=\emptyset$ if $i\neq j$, and $X = \bigsqcup_{i=1}^n A_i$. We collect all $A_i$ of equal cardinality into new sets $B_1,\dots,B_m$. In other words, $A_{i_1}$ and $A_{i_2}$ are both contained in the same $B_j$ if and only if they have the same cardinality. The automorphism $\Delta \colon X \to X$ essentially determines the structure of the given GL-rack $X$. In what follows, we present the definition of a block GL-rack and several propositions characterizing how a GL-rack is determined by $\Delta$.

\begin{definition}
A GL-rack is called a \emph{block GL-rack} if every disjoint short cycle in the cycle decomposition of its diagonal map $\Delta$ has the same length. In other words, if $\Delta = \alpha_1\cdots\alpha_n$, then all $\alpha_i$ have the same length.
\end{definition}

\begin{example}\label{example3.2}
Let \(X = \{1,2,3,4,5,6\}\) with \(u=(12)(35)(46)\) and \(d=\text{id}_X\). The operation table for \(x \ast y\) is given as follows:
    \[
        \begin{array}{c|cccccc}
        \ast & 1 & 2 & 3 & 4 & 5 & 6\\
         \hline
         1 & 2 & 2 & 2 & 2 & 2 & 2\\
         2 & 1 & 1 & 1 & 1 & 1 & 1\\
         3 & 4 & 4 & 5 & 5 & 5 & 5\\
         4 & 5 & 5 & 6 & 6 & 6 & 6\\
         5 & 6 & 6 & 3 & 3 & 3 & 3\\
         6 & 3 & 3 & 4 & 4 & 4 & 4\\
        \end{array}
    \]
where $\Delta=(ud)^{-1}=(12)(35)(46)$. Thus $(X, \ast, u, d)$ is a block GL-rack.
\end{example}

From Example \ref{example3.2}, it can be observed that each block of a block GL-rack seems to be independent and non-interacting with the others. Actually, the following conclusion holds.

\begin{proposition}\label{proposition3.3}
Suppose that $(X, \ast, u, d)$ is a finite block GL-rack with $\Delta = \alpha_1\cdots\alpha_n$. Let $A_i$ be the support of $\alpha_i$ and $\big| A_i \big| = c$. Then 
\begin{enumerate}
\item $x\ast y=x\ast y'$ for any $x\in A_i$ and $y, y'\in A_j$;
\item $x\ast y\neq x'\ast y$ for any $x\neq x'\in A_i$ and $y\in A_j$.
\end{enumerate}
In summary, for any $i, j\in\{1, \cdots, n\}$, $A_i\ast A_j=A_k$ for some $k\in\{1, \cdots, n\}$. In particular $A_i \ast A_i = A_i$ $(1\leq i\leq n)$.
\end{proposition}
\begin{proof}
According to the assumption, $X=\bigsqcup_{i=1}^{n} A_i$ is a block GL-rack with an automorphism $\Delta\colon X\to X$ satisfying $\Delta^c=\operatorname{id}$ and $\Delta(A_i)=A_i$. For any $x\in A_i$ and $y\in A_j$, the definition of a GL-rack yields
\begin{center}
$x\ast y = x\ast \Delta(y) = x\ast \Delta^2(y) = \cdots = x\ast \Delta^c(y) = x\ast y,$
\end{center} 
which implies $x\ast y=x\ast y'$ for any $x\in A_i$ and $y, y'\in A_j$. If $x\ast y\in A_k$, then $\Delta^k(x)\ast y=\Delta^k(x\ast y)\in A_k$, that is, $A_i\ast y=A_k$. The second result follows. Therefore, $A_i \ast A_j = A_k$ and $A_i \ast A_i = A_i$ since $\Delta(A_i)=A_i$.
\end{proof}

A particularly special conclusion can be obtained when the number of blocks is one, that is, when $\Delta$ is a full permutation, meaning that the length of $\Delta$ equals the cardinality of the set $X$.

\begin{corollary}
When $\Delta$ is a full permutation on $X$, the GL-rack $(X, \ast, u, d)$ is a permutation GL-rack.
\end{corollary}
\begin{proof}
In this case, $ud=\Delta^{-1}$ and $x\ast y=x\ast x=\Delta(x)$. The result follows.
\end{proof}

For a given block GL-rack, if we regard each block $A_i$ as a single entity and denote it by $a_i$, we obtain a new set $\widetilde{X}$. We define operations on $\widetilde{X}$ naturally inherited from that on $X$, then we have the following result.

\begin{theorem}\label{theorem3.5}
Let $X = \bigsqcup_{i=1}^{n} A_i$ be a finite block GL-rack. Define $\tilde{X} = \{a_1,\dots,a_n\}$ and the projection map $\pi:X \to \tilde{X}$ with $\pi(A_i)=a_i$. On $\tilde{X}$, we set
\begin{center}
$a_i \tilde{\ast} a_j = \pi(A_i \ast A_j), \quad \tilde{u}(a_i) = \pi\bigl(u(A_i)\bigr), \quad \tilde{d}(a_i) = \pi\bigl(d(A_i)\bigr).$
\end{center}
Then the set $\{\widetilde{X}, \tilde{\ast}, \tilde{u}, \tilde{d}\}$ is a GL-quandle.
\end{theorem}
\begin{proof}
We first verify that $(\widetilde{X}, \tilde{\ast})$ is a quandle. Proposition \ref{proposition3.3} guarantees that the operation $\tilde{\ast}:\tilde{X}\times\tilde{X}\to\tilde{X}$ is well-defined. By Proposition \ref{proposition3.3}, the following three identities hold:
\begin{enumerate}
\item[(1)] $a_i \tilde{\ast} a_i = \pi(A_i \ast A_i) = \pi(A_i) = a_i$;
\item[(2)] For any $y\in A_j$ and $z\in A_k$, there exists a unique element $x\in X$ such that $x\ast y=z$. Assume $x\in A_i$, it follows that for any $A_j$ and $A_k$, there exists a unique $A_i$ such that $A_i \ast A_j = A_k$. As a consequence, for any $a_j$ and $a_k$, there exists a unique $a_i$ satisfying $a_i \tilde{\ast} a_j = a_k$;
\item[(3)] $(a_i \tilde{\ast}a_j)\tilde{\ast} a_k=\pi\big((A_i \ast A_j)\ast A_k\big)           = \pi\big((A_i\ast A_k)\ast (A_j\ast A_k)\big)=(a_i\tilde{\ast} a_k)\tilde{\ast}(a_j \tilde{\ast} a_k)$.
\end{enumerate}
Hence $(\widetilde{X}, \tilde{\ast})$ is indeed a quandle. 

Next, we show that the definition of $\tilde{u}$ is well-defined. For this, it suffices to show that for any $1\leq i\leq n$, there exists a unique $j\in\{1, \cdots, n\}$ such that $u(A_i)=A_j$. Notice that 
\begin{center}
$\Delta u(x)=u(x)\ast u(x)=u(x)\ast x=u(x\ast x)=u\Delta(x)$
\end{center}
for any $x\in X$, it follows that $\Delta u=u\Delta$. We remark that a more general equality about $\Delta$ can be found in \cite[Proposition 2.16 (A2)]{Ta2025}. Assume $\Delta=\alpha_1\cdots\alpha_n$, since $\Delta=u\Delta u^{-1}$, we have
\begin{center}
$\alpha_1\cdots\alpha_n=\Delta=u\Delta u^{-1}=(u\alpha_1 u^{-1})\cdots(u\alpha_n u^{-1})$.
\end{center}
It follows that $u\alpha_1u^{-1}, \cdots, u\alpha_nu^{-1}$ is a permutation of $\alpha_1, \cdots, \alpha_n$, since the decomposition of $\Delta$ is unique up to order. For any $1\leq i\leq n$, there exists a unique $j\in\{1, \cdots, n\}$ such that $u\alpha_i u^{-1}=\alpha_j$, which means $u(A_i)=A_j$. For $\tilde{d}$, the proof is similar as above.

In order to finish the proof, we verify that $\tilde{u}$ and $\tilde{d}$ are compatible with $(\widetilde{X},\tilde{\ast})$. It is straightforward to check that:
\begin{enumerate}
\item[(4)] $\tilde{u}\tilde{d}(a_i \mathbin{\tilde{\ast}} a_i)  = \tilde{u}\tilde{d}(a_i) = a_i$ since $\tilde{u} \tilde{d} = id_{\widetilde{X}}$. Similarly, we have $\tilde{d}\tilde{u}(a_i \mathbin{\tilde{\ast}} a_i) = a_i$.
\item[(5)] $\tilde{u}(a_i \tilde{\ast} a_j) = \tilde{u}(\pi(A_i \ast A_j)) = \pi(u(A_i \ast A_j)) = \pi(u(A_i) \ast A_j) = \pi(u(A_i)) \tilde{\ast} a_j = \tilde{u}(a_i) \tilde{\ast} a_j$. Similarly, we have $\tilde{d}(a_i \tilde{\ast} a_j) = \tilde{d}(a_i) \tilde{\ast} a_j$.
\item[(6)] $a_i \tilde{\ast} \tilde{u}(a_j) = a_i \tilde{\ast} \pi(u(A_j)) = \pi(A_i \ast u(A_j)) = \pi(A_i \ast A_j) = a_i \tilde{\ast} a_j$. Similarly, we have $a_i \tilde{\ast} \tilde{d}(a_j) = a_i \tilde{\ast} a_j$.
\end{enumerate}
It is clear that both $\tilde{u}$ and $\tilde{d}$ are bijections, hence $\{\widetilde{X}, \tilde{\ast}, \tilde{u}, \tilde{d}\}$ is a GL-quandle.
\end{proof}

We next consider the case of general GL-racks. For the unique decomposition of $\Delta = \alpha_1 \cdots \alpha_n$, as noted earlier, we group together all short cycles of the same length in the decomposition. We collect all $A_i$ of equal cardinality into new sets $B_1,\dots,B_m$. In other words, $A_{j_1}$ and $A_{j_2}$are both contained in the same $B_j$ if and only if they have the same cardinality.

\begin{lemma}\label{lemma3.6}
For each $1\leq j\leq m$, $\bigl(B_j, \ast|_{B_j}, u|_{B_j}, d|_{B_j}\bigr)$ is a GL-rack.
\end{lemma}
\begin{proof}
We verify that the structure $\bigl(B_j, \ast|_{B_j}, u|_{B_j}, d|_{B_j}\bigr)$ satisfies the axioms of a GL-rack.

First, we need to verify that $u|_{B_j}$ and $d|_{B_j}$ are both closed on $B_j$. Note that $B_j$ is a union of some supports with equal cardinality. According to the proof of Theorem \ref{theorem3.5}, both $u$ and $d$ send a cycle to another one with the same length, thus their restrictions to $B_j$ are well-defined endomorphisms of $B_j$. Specifically, $u|_{B_j}$ and $d|_{B_j}$ are bijections on $B_j$. Since $u|_{B_j}$ and $d|_{B_j}$ are the restrictions of $u$ and $d$ to $B_j$, respectively, they naturally satisfy conditions (1)--(3) in Definition \ref{definition2.3}.

It remains to prove that $(B_j, \ast|_{B_j})$ is a rack. More precisely, we only need to verify that the operation $\ast|_{B_j}$ is closed on $B_j$. Take arbitrary $x, y\in B_j$, and consider $x\ast y$, $\Delta(x \ast y)$, $\dots$, $\Delta^c(x\ast y)$, where $c=|A_{j_i}|$ for any $A_{j_i}\subset B_j$. Since $\Delta^k(x\ast y)=\Delta^k(x)\ast y$ for $1\le k\le c$ and the order of $x$ with respect to $\Delta$ equals $c$, it follows that the elements $\Delta^k(x) \ast y$ are pairwise distinct. Since $\Delta^c(x)\ast y=x\ast y$, the order of $x\ast y$ with respect to $\Delta$ is also equal to $c$, and hence $x\ast y\in B_j$.
\end{proof}

\begin{proposition}\label{proposition3.7}
For any set $B_j$ $(1\leq j\leq m)$, we have $B_j\ast X=B_j$.
\end{proposition}
\begin{proof}
For any $b \in B_j$ and $x \in X$, consider $b \ast x$, $\Delta(b \ast x)$, \dots, $\Delta^c(b \ast x)$, where $c=|A_{j_i}|$ for any $A_{j_i}\subset B_j$. Since $\Delta^k( b\ast x) = \Delta^k(b) \ast x$ for $1\le k\le c$ and the order of $b$ with respect to $\Delta$ is $c$, it follows that the elements $\Delta^k(b) \ast x$ are pairwise distinct. Since $\Delta^c(b) \ast x = b \ast x$, the order of $b\ast x$ with respect to $\Delta$ is also $c$, and hence $b\ast x\in B_j$. On the other hand, according to Lemma \ref{lemma3.6}, for any $b'\in B_j$ and $x\in B_j\subseteq X$, there exists $b\in B_j$ such that $b\ast x=b'$. Therefore, we conclude that $B_j\ast X=B_j$.
\end{proof}

\begin{remark}
For the unique decomposition of the diagonal map $\Delta$ of a given finite GL-rack, if there is exactly one cycle of a given length, meaning that $B_j$ contains only a single $A_{j_1}$, then $\bigl(B_j, \ast|_{B_j}, u|_{B_j}, d|_{B_j}\bigr)$ is a permutation GL-rack. If the number of cycles of a given length is at least two, i.e., $B_j$ contains more than one $A_{j_i}$, then $\bigl(B_j, \ast|_{B_j}, u|_{B_j}, d|_{B_j}\bigr)$ is a block GL-rack. Therefore, any given finite GL-rack can be decomposed into a collection of permutation GL-racks and block GL-racks. In particular, this decomposition is unique up to order.
\end{remark}

\begin{example}\label{example3.9}
    Let \(X = \{1,2,3,4,5,6\}\) with \(u=(35)(46)\) and \(d=\text{id}\). The operation table for \(x \ast y\) is given as follows:
    \[
        \begin{array}{c|cccccc}
        \ast & 1 & 2 & 3 & 4 & 5 & 6\\
        \hline
        1 & 1 & 1 & 2 & 2 & 2 & 2\\
        2 & 2 & 2 & 1 & 1 & 1 & 1\\
        3 & 4 & 4 & 5 & 5 & 5 & 5\\
        4 & 5 & 5 & 6 & 6 & 6 & 6\\
        5 & 6 & 6 & 3 & 3 & 3 & 3\\
        6 & 3 & 3 & 4 & 4 & 4 & 4\\
        \end{array}
    \]
    Where $\Delta = (ud)^{-1} = (35)(46)$. Define subsets $A_1 = \{1\}$, $A_2 = \{2\}$, $A_3 = \{3,5\}$, $A_4 = \{4,6\}$ and set $B_1 = \{A_1, A_2\}$ and $B_2 = \{A_3,A_4\}$. Then the GL-rack $(X, \ast, u, d)$ can be regarded as being spliced by block GL-racks $\bigl(B_1, \ast|_{B_1}, u|_{B_1}, d|_{B_1}\bigr)$ and $\bigl(B_2, \ast|_{B_2}, u|_{B_2}, d|_{B_2}\bigr)$.
\end{example}

\subsection{The coloring invariants of Legendrian knots}
In this subsection, we clarify the relationship between the number of colorings of a Legendrian knot with respect to a given finite GL-rack and the number of colorings with respect to the permutation GL-racks and block GL-racks obtained from its decomposition. We also relate the number of colorings of a Legendrian knot with respect to a given finite block GL-rack to the number of colorings with respect to the GL-quandle induced from this block GL-rack.

\begin{theorem}\label{theorem3.10}
Let $K$ be a Legendrian knot, and $(X, \ast, u, d)$ a finite GL-rack that decomposes into a disjoint union of block GL-racks $\bigl(B_j, \ast|_{B_j}, u|_{B_j}, d|_{B_j}\bigr)$, then $\operatorname{Col}_X(K) =\sum_j \operatorname{Col}_{B_j}(K)$.
\end{theorem}

\begin{proof}
    Consider $f\in\operatorname{Hom}(\operatorname{GLR}(K),X)$, where $\operatorname{GLR}(K)$ is generated by arcs $x_i$ subject to relations $r_i: u^{p_i}d^{q_i}(x_i)\ast^{\varepsilon_i} x_{k_i}=x_{i+1}$. Here $p_i$ and $q_i$ denote the number of upward cusps and downward cusps between $x_i$ and $x_{i+1}$, and $\varepsilon_i$ denotes the sign of the crossing before $x_{i+1}$. Note that the fact $ud=du$ \cite{GL-rack-Karmakar-Saraf-Singh} is used here. If $f(x_1)\in B_j$ for some $j$, then by Proposition \ref{proposition3.7} we have $f(x_2)=f(x_1) \ast f(x_{k_1})\in B_j$. By induction, all $f(x_i)$ belong to the same $B_j$. Lemma \ref{lemma3.6} shows that each $\bigl(B_j, \ast|_{B_j}, u|_{B_j}, d|_{B_j}\bigr)$ is a GL-rack, so $f\in\operatorname{Hom}(\operatorname{GLR}(K),B_j)$. Thus $\operatorname{Col}_X(K)=\sum_j \operatorname{Col}_{B_j}(K)$.
\end{proof}

For a Legendrian knot $K$ and a given finite block GL-rack $(X,\ast,u,d)$, we analyze the relation between colorings of $K$ with respect to $(X,\ast,u,d)$ and colorings of its projection $(\widetilde{X}, \tilde{\ast}, \tilde{u}, \tilde{d})$. Consider the following diagram of maps.
\[
\begin{tikzcd}
    & \operatorname{GLR}(K) \arrow[r,"\phi"] \arrow[rd,"\psi"] & X \arrow[d,"\pi"] \\
    & & \tilde{X}
\end{tikzcd}
\]
If $\phi\in\operatorname{Hom}(\operatorname{GLR}(K), X)$ satisfies $\psi=\pi\circ\phi\in\operatorname{Hom}(\operatorname{GLR}(K), X)$, then $\phi$ is called a \textit{lift} of $\psi$. We denote the set of all lifts of $\psi$ by $\operatorname{Lift}(\psi)$. Then we obtain the following result.

\begin{theorem}\label{theorem3.11}
Let $K$ be a Legendrian knot and $(X,\ast,u,d)$ a finite block GL-rack, we use $c$ to denote the size of each block. If $\phi \in \operatorname{Hom}(\operatorname{GLR}(K), X)$, then $\psi=\pi\circ\phi\in\operatorname{Hom}(\operatorname{GLR}(K), \widetilde{X})$. Furthermore, we have
    \[
    \operatorname{Col}_X(K) = \sum_{\psi \in \operatorname{Hom}(\operatorname{GLR}(K), \widetilde{X})} \big|\operatorname{Lift}(\psi)\big|,
    \]
    and $\big|\operatorname{Lift}(\psi)\big|$ can only be either $0$ or $c$.
\end{theorem}
\begin{proof}
    It is clear that $\psi = \pi\circ\phi \in \operatorname{Hom}(\operatorname{GLR}(K), \widetilde{X})$. Hence, every $\phi$ is a lift of some $\psi$, so
    \[
    \operatorname{Col}_X(K) = \sum_{\psi \in \operatorname{Hom}(\operatorname{GLR}(K), \widetilde{X})} \big|\operatorname{Lift}(\psi)\big|.
    \]
    
    Choose an element $\psi\colon\operatorname{GLR}(K)\to\widetilde{X}$ with $|\operatorname{Lift}(\psi)|\neq0$, in order to complete the proof, it is sufficient to show that in this case we have $|\operatorname{Lift}(\psi)|=c$. To this end, suppose that the fundamental GL-rack of $K$ is given by
    \[
    \operatorname{GLR}(K) = \langle x_1,\dots,x_n \mid r_1,\dots,r_n \rangle.
    \]
    Here, the relation $r_i$ has the following form
    \[
    u^{p_i}d^{q_i}(x_i) \ast^{\varepsilon_i} x_{k_i} = x_{i+1},
    \]
    here $x_{n+1} = x_1$. We suppose that $\psi(x_i) = a_{j_i} = \pi(A_{j_i})$. If $\phi$ is a lift of $\psi$, then we must have $\phi(x_i) \in A_{j_i}$. For a fixed $\phi(x_1)$, since $\phi(x_{k_1}) \in A_{j_{k_1}}$, we have
    \[
    u^{p_1}d^{q_1}(\phi(x_1)) \ast ^{\varepsilon_1} \phi(x_{k_1}) = u^{p_1}d^{q_1}(\phi(x_1)) \ast ^{\varepsilon_1} A_{j_{k_1}}.
    \]
    Thus, $\phi(x_2)$ is completely determined by $\phi(x_1)$ and $\psi$, since $A_{j_{k_1}}$ is determined by $\psi(x_{k_1})$. Similarly, all $\phi(x_i)$ $(1\leq i\leq n)$ are determined, and the relations $r_1,\dots,r_{n-1}$ are satisfied. For the last relation $r_n$, since 
    \begin{center}
    $\tilde{u}^{p_n}\tilde{d}^{q_n}(\psi(x_n))\tilde\ast^{\varepsilon_n}\psi(x_{k_n})=\psi(x_1)$,
    \end{center}
    and $\phi$ is a lift of $\psi$, we have
    \begin{center}
    $u^{p_n}d^{q_n}(\phi(x_n))\ast^{\varepsilon_n}\phi(x_{k_n})=\phi(x_1)$,
    \end{center}  
    which means that $r_n$ also holds. Therefore, in order to calculate $|\operatorname{Lift}(\psi)|$, it suffices to count the number of $\phi(x_1)$ for a fixed $\psi$. For a fixed lift $\phi$, consider $\Delta^k (\phi(x_1))$ $(1 \leq k \leq c)$. Note that for distinct $k$ and $k'$, $\Delta^k(\phi(x_1))\neq\Delta^{k'}(\phi(x_1))$. Since $\phi$ is a coloring, together with the fact that $\Delta$ is a rack automorphism, one concludes that $\Delta^k\phi$ is also a coloring and a lift of $\psi$. There are exactly $c$ choices for $\phi(x_1)$, we conclude that $|\operatorname{Lift}(\psi)|=c$.
\end{proof}

\begin{corollary}
Let $(X,\ast,u,d)$ be a finite block GL-rack and $c$ the size of each block of it, then for any Legendrian knot $K$, the coloring invariant Col$_X(K)$ is divided by $c$.
\end{corollary}

\begin{example}
    Consider coloring invariant of the following Legendrian trefoil knot.
    \begin{figure}[h]
        \centering
        \tikzset{every picture/.style={line width=0.75pt}} 
        \tikzset{every picture/.style={line width=0.75pt}} 
        \begin{tikzpicture}[x=0.75pt,y=0.75pt,yscale=-1,xscale=1]

\draw    (144.6,100.4) .. controls (184.6,70.4) and (158.6,142) .. (198.6,112) ;
\draw    (174.2,100.8) .. controls (178.57,97.52) and (182.26,95.55) .. (185.36,94.73) .. controls (191.04,93.22) and (194.73,95.57) .. (197,100.8) .. controls (200.5,108.89) and (207.8,116.8) .. (217,116.8) ;
\draw    (122.6,86.8) .. controls (152.2,85.6) and (126.87,137.2) .. (166.87,107.2) ;
\draw    (136.6,104.8) .. controls (133.4,107.6) and (127.4,120.4) .. (119.8,120.4) ;
\draw    (119.8,120.4) .. controls (141,121.6) and (146.2,143.2) .. (169.8,143.2) .. controls (193.4,143.2) and (199,115.6) .. (217,116.8) ;
\draw    (122.6,86.8) .. controls (143.88,85.16) and (149.8,63.6) .. (168.6,64) .. controls (187.4,64.4) and (186.2,80.4) .. (216.2,81.6) ;
\draw    (204.6,106) .. controls (211,95.2) and (205.4,81.6) .. (216.2,81.6) ;
\draw   (161.99,61.88) .. controls (164.19,63.13) and (166.35,63.82) .. (168.46,63.93) .. controls (166.4,64.4) and (164.39,65.45) .. (162.43,67.06) ;
\draw   (162.79,140.19) .. controls (164.77,141.77) and (166.8,142.78) .. (168.86,143.22) .. controls (166.76,143.36) and (164.61,144.08) .. (162.43,145.37) ;

\draw (151.07,115.6) node [anchor=north west][inner sep=0.75pt]   [align=left] {$\displaystyle x_{1}$};
\draw (175.27,106.8) node [anchor=north west][inner sep=0.75pt]   [align=left] {$\displaystyle x_{2}$};
\draw (112.47,100.6) node [anchor=north west][inner sep=0.75pt]   [align=left] {$\displaystyle x_{3}$};
\end{tikzpicture}
        \caption{Legendrian trefoil K}
        \label{fig4}
    \end{figure}
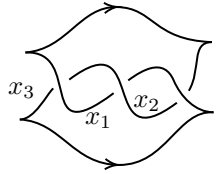
    \begin{enumerate}
        \item For the GL-rack given in example \ref{example3.9}, it is straightforward to verify that
        \[
        \operatorname{Col}_X(K) = \operatorname{Col}_{B_1}(K) + \operatorname{Col}_{B_2}(K) = 2 +  0 = 2.
        \]
        \item For the GL-rack given in example \ref{example3.2}, let us denote $A_1=\{1, 2\}$, $A_2=\{3, 5\}$, $A_3=\{4, 6\}$. Then we can define a GL-quandle $\widetilde{X} = \{a_1, a_2, a_3\}$ with $\tilde{u}=\tilde{d}=\text{id}_{\tilde{X}}$ and the operation table for \(x\tilde{\ast}y\) is given as follows:
        \[
        \begin{array}{c|ccc}
        \tilde{\ast} & a_1 & a_2 & a_3 \\
        \hline
        a_1 & a_1 & a_1 & a_1 \\
        a_2 & a_3 & a_2 & a_2 \\
        a_3 & a_2 & a_3 & a_3 \\
        \end{array}
        \]
        The projection $\pi:X \to \tilde{X}$ maps $A_i$ to $a_i$. We investigate colorings of the Legendrian trefoil knot $K$ associated with $X$ and its projection $\widetilde{X}$. It is readily seen that $\operatorname{Col}_{\widetilde{X}}(K) = 3$. Furthermore, for $\psi_j \in \operatorname{Hom}(\operatorname{GLR}(K), \widetilde{X})$ with $j = 1,2,3$, we have $\psi_j(x_i) = a_j$. Next, consider $\phi_1 \in \operatorname{Hom}(\operatorname{GLR}(K), X)$ such that $\phi_1(x_i) \in A_1$. Fix $\phi_1(x_1) = 1$, then
        \[
        \phi_1(x_2) = u d(1) \ast A_1 = 2 \ast A_1 = 1,
        \]
        and
        \[
        \phi_1(x_3) = \phi_1(x_2) \ast \phi_1(x_1) = 1 \ast 1 = 2.
        \]
        However, $\phi_1(x_1), \phi_1(x_2), \phi_1(x_3)$ need to satisfy the relation
        \[
        u d(\phi_1(x_3)) \ast \phi_1(x_2) = \phi(x_1),
        \]
        which becomes
        \[
        1 \ast 1 = 1,
        \]
        a contradiction. Hence $\operatorname{Lift}(\psi_1) = 0$. Similarly, we find that $\operatorname{Lift}(\psi_2)=\operatorname{Lift}(\psi_3)=0$, which follows that $\operatorname{Col}_X(K)=0$.
    \end{enumerate}
\end{example}

\begin{remark}
The operation $\ast$ restricts to a well-defined operation on each $A_i$, and $(A_i, \ast|_{A_i})$ is a permutation rack. Nevertheless, it is generally incompatible with $u$ and $d$, which implies that an individual $A_i$ does not carry a GL-rack structure in general. In particular, the coloring number with respect to a finite block GL-rack $X=\bigsqcup_{i=1}^nA_i$ cannot be expressed as the sum of coloring numbers of its permutation racks $A_i$.
\end{remark}

\section{Coloring invariants of Legendrian knots with the same classical invariants}\label{section4}
In this section, we prove that if two Legendrian knots have the same classical invariants, that is, they are of the same knot type and have equal Thurston-Bennequin number and rotation number, then their coloring invariants with respect to any given finite GL-rack coincide. The key to the proof is Theorem \ref{theorem3.10} in Section 3: every finite GL-rack decomposes into permutation GL-racks and block GL-racks. Thus it suffices to show that the two Legendrian knots have equal coloring invariant with respect to these two types of GL-racks.

First, we state a result due to Kimura, who proved that two Legendrian knots with identical classical invariants share the same coloring invariants with respect to any given finite GL-quandle. The key to his proof relies on Theorem \ref{theorem2.1}: two Legendrian knots $K_1$ and $K_2$ have identical classical invariants if and only if there exists a positive integer $N$ such that $K_1$ and $K_2$ become Legendrian isotopic after $N$ times positive stabilization and $N$ times negative stabilization. In other words, $S_+^N S_-^N(K_1)$ and $S_+^N S_-^N(K_2)$ are Legendrian isotopic. This theorem will also be used in the proofs of our subsequent results.

\begin{theorem}\label{theorem4.1}\cite{Bi-Legendrian-2023}
Let $K_1$ and $K_2$ be two Legendrian knots of the same knot type. If $tb(K_1)=tb(K_2)$ and $rot(K_1)=rot(K_2)$, then for any finite GL-quandle $(X, \ast, u, d)$, we have $\operatorname{Col}_X(K_1)=\operatorname{Col}_X(K_2)$.
\end{theorem}

We first prove that two Legendrian knots with identical classical invariants possess the same coloring invariants with respect to a given finite permutation GL-rack. The key step of the proof relies essentially on the fact that the rack operation $x\ast y=x\ast x=\Delta(x)$ is independent of $y$.

\begin{theorem}\label{theorem4.2}
Let $K_1$ and $K_2$ be two Legendrian knots of the same knot type. If $tb(K_1) = tb(K_2) = t$ and $rot(K_1) = rot(K_2) = r$, then for any finite permutation GL-rack $(X, \ast, u, d)$, we have $\operatorname{Col}_X(K_1) = \operatorname{Col}_X(K_2)$.
\end{theorem}

\begin{proof}
    Suppose that the fundamental GL-racks of $K_1$ and $K_2$ are given by
    \begin{center}
        $\operatorname{GLR}(K_1) = \langle x_1,\dots,x_n \mid r_1,\dots,r_n \rangle$
    \end{center}
    and
    \begin{center}
        $\operatorname{GLR}(K_2) = \langle y_1,\dots,y_m \mid \tilde{r}_1,\dots,\tilde{r}_m \rangle,$
    \end{center}
    respectively, with the corresponding generators and relations.The above notation will be retained throughout the sequel without further ambiguity.
    
    For a finite permutation GL-rack $(X, \ast, u, d)$, any coloring $f$ of the Legendrian knot $K_1$ is uniquely determined by its value on a single arc. That is, if $f \colon \operatorname{GLR}(K_1) \to X$ is a coloring, then $f(x_1)$ uniquely determines all $f(x_i)$ for $i = 1, 2, \dots, n$. The same holds for the Legendrian knot $K_2$. Due to the relations $r_1,\dots,r_n$ in $GLR(K_1)$, for $f$ being a homomorphism, it satisfies
    \begin{center}
        $f(x_1) = u^p d^q \sigma^\omega f(x_1).$
    \end{center}
    Correspondingly, for $K_2$, if $g:GLR(K_2) \to X$ is a coloring, we have
    \begin{center}
        $g(y_1) = u^{\tilde{p}} d^{\tilde{q}} \sigma^{\tilde{\omega}} g(y_1).$
    \end{center}
    Here $p$, $q$, and $\omega$ (resp.\ $\tilde{p}$, $\tilde{q}$, and $\tilde{\omega}$) denote the number of up cusps, down cusps, and the writhe of the Legendrian knot $K_1$ (resp.\ $K_2$). Since
    \begin{center}
        $t = \omega - \frac{1}{2}(p+q) = \tilde{\omega} - \frac{1}{2}(\tilde{p}+\tilde{q}), \qquad r = \frac{1}{2}(q-p) = \frac{1}{2}(\tilde{q}-\tilde{p}).$
    \end{center}
    and $ud=\sigma^{-1}$, together with the fact that $ud=du$, we obtain
    \begin{center}
        $f(x_1) = u^{-\frac{1}{2}(q-p)} d^{\frac{1}{2}(q-p)} u^{\frac{1}{2}(p+q)} d^{\frac{1}{2}(p+q)} \sigma^\omega f(x_1) = u^{-r} d^{r} \sigma^{t} f(x_1).$
    \end{center}
    Similarly,
    \begin{center}
        $g(y_1) = u^{-\frac{1}{2}(\tilde{q}-\tilde{p})} d^{\frac{1}{2}(\tilde{q}-\tilde{p})} u^{\frac{1}{2}(\tilde{p}+\tilde{q})} d^{\frac{1}{2}(\tilde{p}+\tilde{q})} \sigma^{\tilde{\omega}} g(y_1) = u^{-r} d^{r} \sigma^{t} g(y_1).$
    \end{center}
    Thus, the coloring $f$ (resp. $g$) depends only on $t$ and $r$. Since
    \begin{center}
        $tb(K_1) = tb(K_2) = t$ and $rot(K_1) = rot(K_2) = r$,
    \end{center}
    the numbers of colorings of $K_1$ and $K_2$ are equal.
\end{proof}

Next, we prove that two Legendrian knots with identical classical invariants have the same coloring invariants with respect to a given finite block GL-rack. For convenience of the proof, we first state a lemma, which discusses the relation between the coloring invariant of a Legendrian knot $K$ and that of the Legendrian knot $S_+^NS_-^N(K)$, with respect to the given finite block GL-rack. Since the location of stabilization does not affect Legendrian knot type, we assume throughout that all the stabilizations are performed at the arc $x_1$. In other words, if the crossing between $x_1$ and $x_2$ in $K$ corresponds to the relation $u^{p_1}d^{q_1}(x_1)\ast^{\varepsilon_1} x_{k_1}=x_2$, then this crossing in $S_+^NS_-^N(K)$ corresponds to the relation $u^{2N+p_1}d^{2N+q_1}(x_1) \ast^{\varepsilon_1}x_{k_1}=x_2$.

\begin{lemma}\label{lemma4.3}
    Let $K$ be a Legendrian knot, and let $S_+^N S_-^N(K)$ be the Legendrian knot obtained from $K$ by $N$ positive stabilizations and $N$ negative stabilizations. Let $(X, \ast, u, d)$ be a finite block GL-rack, assume $\Delta = \alpha_1\cdots\alpha_l$ with support $A_i = \operatorname{supp}(\alpha_i)$ $(1\le i\le l)$. Let $(\widetilde{X}=\{a_1, \cdots, a_l\}, \tilde{\ast}, \tilde{u}, \tilde{d})$ be the induced GL-quandle of $(X, \ast, u, d)$. Then $\operatorname{Hom}(\operatorname{GLR}(K),\widetilde{X}) = \operatorname{Hom}(\operatorname{GLR}(S_+^N S_-^N(K)),\widetilde{X})$. Furthermore, for any $\psi\in\operatorname{Hom}(\operatorname{GLR}(K),\widetilde{X})$, there is a $\psi'\in\operatorname{Hom}\bigl(\operatorname{GLR}(S_+^N S_-^N(K)),\widetilde{X}\bigr)$ such that $\psi'(x_i)=\psi(x_i)$. Moreover, if $|\operatorname{Lift}(\psi)|\neq 0$, then $|\operatorname{Lift}(\psi')|\neq 0$ if and only if $\Delta^{2N}=\operatorname{id}_X$.
\end{lemma}
\begin{proof}
Suppose that
\begin{center}
$\operatorname{GLR}(K) = \langle x_1,\dots,x_n \mid r_1,\dots,r_n \rangle, \quad \operatorname{GLR}(S_+^N S_-^N(K)) = \langle x_1,\dots,x_n \mid r_1',\dots,r_n' \rangle,$
\end{center}
where each relation $r_i$ is given by
\begin{center}
$u^{p_i} d^{q_i} (x_i) *^{\varepsilon_i} x_{k_i} = x_{i+1} \quad (1\leq i\leq n)$,
\end{center}
here $x_{n+1}=x_1$. Among the relations $r_i'$, only $r_1'$ is changed to
\begin{center}
$u^{2N+p_1} d^{2N+q_1}(x_1) \ast^{\varepsilon_1} x_{k_1} = x_2,$
    \end{center}
while any other $r_i'$ $(2\leq i\leq n)$ coincide with $r_i$. Since $(\widetilde{X}, \widetilde{*}, \widetilde{u}, \widetilde{d})$ is a GL-quandle, it follows naturally that $\widetilde{u}\widetilde{d}=\operatorname{id}_{\widetilde{X}}$, and hence $\operatorname{Hom}(GLR(K),\widetilde{X}) = \operatorname{Hom}(GLR(S_+^N S_-^N(K)),\widetilde{X})$. Furthermore, for any $\psi\in\operatorname{Hom}(\operatorname{GLR}(K),\widetilde{X})$, there is a $\psi'\in\operatorname{Hom}\bigl(\operatorname{GLR}(S_+^N S_-^N(K)),\widetilde{X}\bigr)$ such that $\psi'(x_i)=\psi(x_i)$.

If $|\operatorname{Lift}(\psi)|\neq 0$, then there exists $\phi\in\operatorname{Hom}(\operatorname{GLR}(K),X)$ such that $\phi\circ\pi=\psi$ and $\phi(x_i)$ $(1\leq i\leq n)$ satisfy the relations $r_1,\dots,r_n$. If $\operatorname{Lift}(\psi')\neq\emptyset$, then we choose a lift $\phi'\in\operatorname{Lift}(\psi')$. Now $\psi'(x_i)=\psi(x_i)$ implies that $\phi(x_i)$ and $\phi'(x_i)$ must both lie in $\pi^{-1}(\psi(x_i))$. Without loss of generality, let us assume $\phi'(x_1)=\phi(x_1)$. If not, we can use $\Delta^k\phi'$ to replace $\phi'$ for some suitable $k$. Now the relation $r_1'$ gives
\begin{center}
$\Delta^{-2N}(u^{p_1}d^{q_1}\phi'(x_1))\ast^{\varepsilon_1}\phi'(x_{k_1})=\phi'(x_2)$. 
\end{center}
Since
\begin{center}
$u^{p_1}d^{q_1}\phi(x_1)\ast^{\varepsilon_1}\phi(x_{k_1})=\phi(x_2)$
\end{center}
and $\phi'(x_{k_1})$ and $\phi(x_{k_1})$ lie in the same block $\pi^{-1}(\psi(x_{k_1}))$, together with Proposition \ref{proposition3.3}, we conclude that 
\begin{align*}
\phi'(x_2)&=\Delta^{-2N}(u^{p_1}d^{q_1}\phi'(x_1))\ast^{\varepsilon_1}\phi'(x_{k_1})\\
&=\Delta^{-2N}(u^{p_1}d^{q_1}\phi(x_1))\ast^{\varepsilon_1}\phi(x_{k_1})\\
&=\Delta^{-2N}(u^{p_1}d^{q_1}\phi(x_1))\ast^{\varepsilon_1}\Delta^{-2N}(\phi(x_{k_1}))\\
&=\Delta^{-2N}((u^{p_1}d^{q_1}\phi(x_1))\ast^{\varepsilon_1}\phi(x_{k_1}))\\
&=\Delta^{-2N}(\phi(x_2)).
\end{align*}
Similarly, one obtains that
\begin{center}
$\phi'(x_i)=\Delta^{-2N}(\phi(x_i)) \quad (2\leq i \leq n)$.
\end{center}
The relation $r_n'$ reads
\begin{center}
$\phi'(x_n) \ast^{\varepsilon_n} \phi'(x_{k_n})=\phi'(x_1),$
\end{center}
which now becomes
\begin{center}
$\Delta^{-2N}(\phi(x_n)\ast^{\varepsilon_n}\phi(x_{k_n}))=\phi(x_1).$
\end{center}
On the other hand, the relation $r_n$ yields
\begin{center}
$\phi(x_n) \ast^{\varepsilon_n} \phi(x_{k_n})=\phi(x_1).$
\end{center}
Hence we must have
\begin{center}
$\Delta^{-2N}(\phi(x_1))=\phi(x_1).$
\end{center}
Since $\Delta$ is a cyclic permutation on $\pi^{-1}(\psi(x_1))$, it follows that $\Delta^{2N}=\operatorname{id}_{\pi^{-1}(\psi(x_1))}=\operatorname{id}_X$.
    
Conversely, if $\Delta^{2N}=\operatorname{id}_X$, then we have $\operatorname{GLR}(K)\cong\operatorname{GLR}(S_+^N S_-^N(K))$ and hence
\begin{center}
$\operatorname{Hom}(\operatorname{GLR}(K),X) = \operatorname{Hom}(\operatorname{GLR}(S_+^N S_-^N(K)),X),$
\end{center}
so $|\operatorname{Lift}(\psi')|=|\operatorname{Lift}(\psi)|\neq 0$.
\end{proof}

Lemma \ref{lemma4.3} implies that if $\Delta^{2N} = \operatorname{id}_X$, then $|\operatorname{Lift}(\psi)| = |\operatorname{Lift}(\psi')|$; if $\Delta^{2N} \neq \operatorname{id}_X$, then at least one of $|\operatorname{Lift}(\psi)|$ and $|\operatorname{Lift}(\psi')|$ is zero. Based on this, we now prove that two Legendrian knots with identical classical invariants have the same coloring invariants with respect to a given finite block GL-rack.

\begin{theorem}\label{theorem4.4}
Let $K_1$ and $K_2$ be two Legendrian knots of the same knot type. If $tb(K_1)=tb(K_2)=t$ and $rot(K_1)=rot(K_2)=r$, then for any finite block GL-rack $(X, \ast, u, d)$, we have $\operatorname{Col}_X(K_1)=\operatorname{Col}_X(K_2)$.
\end{theorem}
\begin{proof}
Suppose that the fundamental GL-racks of $K_1$ and $K_2$ are given by
\[
\operatorname{GLR}(K_1)=\langle x_1,\dots,x_n \mid r_1,\dots,r_n \rangle
\]
and
\[
\operatorname{GLR}(K_2)=\langle y_1,\dots,y_m \mid \tilde{r}_1,\dots,\tilde{r}_m \rangle,
\]
And suppose that the fundamental GL-racks of $S_+^NS_-^N(K_1)$ and $S_+^NS_-^N(K_2)$ are given by
\[
\operatorname{GLR}(S_+^NS_-^N(K_1))=\langle x_1, \dots, x_n \mid r_1', \cdots,r_n' \rangle
\]
and
\[
\operatorname{GLR}(S_+^NS_-^N(K_2))=\langle y_1, \dots, y_m \mid \tilde{r}_1', \cdots, \tilde{r}_m' \rangle.
\]
Here, the relations $r_i$ are of the form
\[
u^{p_i}d^{q_i}(x_i)\ast^{\varepsilon_i} x_{k_i} = x_{i+1},
\]
and the relations $\tilde{r}_i$ are of the form
\[
u^{\tilde{p}_i}d^{\tilde{q}_i}(y_i)\ast^{\tilde{\varepsilon}_i}y_{k_i}=y_{i+1},
\]
where $x_{n+1}=x_1$ and $y_{m+1}=y_1$. All relations $r_i'$ coincide with $r_i$ for $2\le i\le n$, except that $r_1'$ differs from $r_1$. Similarly, all relations $\tilde{r}_i'$ coincide with $\tilde{r}_i$ for $2\le i\le m$, except that $\tilde{r}_1'$ differs from $\tilde{r}_1$. Here all relations are expressed as the same as that in Lemma \ref{lemma4.3}, and all notations used are the same as that in Lemma \ref{lemma4.3}.

Since $K_1$ and $K_2$ are of the same knot type and have the same Thurston-Bennequin number and rotation number, Theorem \ref{theorem2.1} implies that there exists a positive integer $N$ such that $S_+^N S_-^N(K_1)$ and $S_+^N S_-^N(K_2)$ are Legendrian isotopic. Hence, there is an isomorphism $h \colon \operatorname{GLR}(S_{+}^N S_{-}^N(K_1)) \to \operatorname{GLR}(S_{+}^N S_{-}^N(K_2))$. Consider the following commutative diagram.
\[
\begin{tikzcd}
    & \operatorname{GLR}(S_+^N S_-^N(K_1)) \arrow[r,"\phi_1'"] \arrow[rd,"\psi_1'"] & X \arrow[d,"\pi"] & \operatorname{GLR}(S_+^N S_-^N(K_2)) \arrow[l,"\phi_2'"] \arrow[ld,"\psi_2'"] \\
    & & \tilde{X} & &
\end{tikzcd}
\]
For any $\psi_2' \in \operatorname{Hom}\bigl(\operatorname{GLR}(S_+^N S_-^N(K_2)), \widetilde{X}\bigr)$, there exists $\psi_1' = \psi_2' \circ h \in \operatorname{Hom}\bigl(\operatorname{GLR}(S_+^N S_-^N(K_1)), \widetilde{X}\bigr)$. If $\phi_2' \in \operatorname{Hom}\bigl(\operatorname{GLR}(S_+^NS_-^N(K_2)), X\bigr)$ is a lift of $\psi_2'$, then $\phi_1' = \phi_2' \circ h \in \operatorname{Hom}\bigl(\operatorname{GLR}(S_+^N S_-^N(K_1)), X\bigr)$ is a lift of $\psi_1'$. It follows that $|\operatorname{Lift}(\psi_1')|\geq|\operatorname{Lift}(\psi_2')|$. Similarly, we can also obtain $|\operatorname{Lift}(\psi_1')|\leq|\operatorname{Lift}(\psi_2')|$, which implies $|\operatorname{Lift}(\psi_1')| = |\operatorname{Lift}(\psi_2')|$. 

By Lemma \ref{lemma4.3}, for $\psi_1'$ and $\psi_2'$, there exist unique $\psi_1 \in \operatorname{Hom}\bigl(\operatorname{GLR}(K_1),\widetilde{X}\bigr)$ and $\psi_2 \in \operatorname{Hom}\bigl(\operatorname{GLR}(K_2),\widetilde{X}\bigr)$ such that $\psi_1(x_i) = \psi_1'(x_i)$ and $\psi_2(y_j) = \psi_2'(y_j)$ for $1\le i\le n$ and $1\le j\le m$. We next prove that if some $\psi_1 \in \operatorname{Hom}(\operatorname{GLR}(K_1),\widetilde{X})$ satisfies $|\operatorname{Lift}(\psi_1)| \neq 0$, then for the corresponding $\psi_1' \in \operatorname{Hom}\bigl(\operatorname{GLR}(S_+^N S_-^N(K_1)),\widetilde{X}\bigr)$ and $\psi_2' = \psi_1' \circ h^{-1} \in \operatorname{Hom}\bigl(\operatorname{GLR}(S_+^N S_-^N(K_2)),\widetilde{X}\bigr)$, the associated $\psi_2 \in \operatorname{Hom}(\operatorname{GLR}(K_2),\widetilde{X})$ must satisfy $|\operatorname{Lift}(\psi_2)| \neq 0$.

According to Lemma \ref{lemma4.3}, we proceed with the following discussion by considering the following two cases:
\begin{itemize}
\item if $\Delta^{2N}=\operatorname{id}_X$, then we have
\[
|\operatorname{Lift}(\psi_2)|=|\operatorname{Lift}(\psi_2')|=|\operatorname{Lift}(\psi_1')|=|\operatorname{Lift}(\psi_1)| \neq 0.
\]
\item If $\Delta^{2N}\neq\operatorname{id}_X$, then both $|\operatorname{Lift}(\psi_1')|$ and $|\operatorname{Lift}(\psi_2')|$ vanish. Nevertheless, there exists a positive integer $k$ such that $\Delta^{2N+2k} = \mathrm{id}_X$. For example, any positive integer $k$ satisfying $c|2N+2k$ meets the requirement, here $c$ denotes the size of each block. In this case, the lifts with respect to the coloring coincide for $S_+^{N+k} S_-^{N+k}(K_i)$ and $K_i$ ($i=1,2$). Moreover, $S_+^{N+k} S_-^{N+k}(K_1)$ and $S_+^{N+k} S_-^{N+k}(K_2)$ remain Legendrian isotopic. Therefore $|\mathrm{Lift}(\psi_2)| \neq 0$.
\end{itemize}
For every $\psi_1 \in \operatorname{Hom}(\operatorname{GLR}(K_1),\widetilde{X})$ with $|\operatorname{Lift}(\psi_1)| \neq 0$, we can find a $\psi_2 \in \operatorname{Hom}(\operatorname{GLR}(K_2),\widetilde{X})$ such that $|\operatorname{Lift}(\psi_2)| \neq 0$. Since the number of lifts takes only two values, $0$ or $|A_i|$, it follows that there exists a one-to-one correspondence between $\operatorname{Hom}(\operatorname{GLR}(K_1),\widetilde{X})$ and $\operatorname{Hom}(\operatorname{GLR}(K_2),\widetilde{X})$ such that $|\operatorname{Lift}(\psi_1)|\neq0$ if and only if the corresponding $|\operatorname{Lift}(\psi_2)|\neq0$. Therefore,
\[
\operatorname{Col}_X(K_1) = \sum_{\psi_1 \in \operatorname{Hom}(\operatorname{GLR}(K_1),\widetilde{X})} |\operatorname{Lift}(\psi_1)| = \sum_{\psi_2 \in \operatorname{Hom}(\operatorname{GLR}(K_2),\widetilde{X})} |\operatorname{Lift}(\psi_2)| = \operatorname{Col}_X(K_2).
\]
The proof is finished.
\end{proof}

By combining Theorem \ref{theorem3.10}, Theorem \ref{theorem4.2}, and Theorem \ref{theorem4.4}, we complete the proof of Theorem \ref{theorem1.1}.

\begin{remark}
Here we list some remarks about the results above.
\begin{enumerate}
\item In \cite{Bi-Legendrian-2023}, Kimura proved that two Legendrian knots with identical classical invariants have the same coloring invariant with respect to any given finite GL-quandle. Theorem \ref{theorem1.1} extends this result from finite GL-quandles to finite GL-racks.
\item We would like to remark that, similar to Theorem \ref{theorem4.1}, it is possible to prove Theorem \ref{theorem1.1} directly by using the result of Theorem \ref{theorem2.1}. This is due to the fact that $\Delta$ is of finite order, and it suffices to take $N$ sufficiently large such that $\Delta^{2N}=\operatorname{id}$. However, such an argument ignores the decomposition structure of the GL-rack $X$.
\item By distinguishing left cusp and right cusp, the notion of GL-rack can be generalized to 4-Legendrian rack, which was first studied by Naoki Kimura in his PhD Thesis. Very recently, Ta and Wood proved an analogue of Theorem \ref{theorem4.2} for permutation 4-Legendrian racks \cite{TW2026}. Moreover, similar to the main result of \cite{Fundamental-GL-rack-and-classical-invariants}, which states that the fundamental GL-rack determines the classical invariants up to the sign, in \cite{TW2026} Ta and Wood proved that the fundamental 4-Legendrian rack completely determines the classical invariants.
\item For the fundamental GL-rack $\operatorname{GLR}(K)$ of a Legendrian knot $K$, now we know that it determines the Thurston-Bennequin number $tb(K)$ and rotation number $rot(K)$ in the sense of ignoring the signs \cite{Fundamental-GL-rack-and-classical-invariants}. On the other hand, if we are given two Legendrian knots with the same classical invariants, Theorem \ref{theorem1.1} tells us that for any finite GL-rack $X$, these two Legendrian knots have the same coloring invariant. First, it is natural to investigate the difference between $\operatorname{GLR}(K)$ and the set $\{\operatorname{Col}_K(X)\}$ for all finite GL-rack $X$. For knot groups, if for any finite group $G$ we have $|\operatorname{Hom}(\pi_1(S^3-K_1), G)|=|\operatorname{Hom}(\pi_1(S^3-K_2), G)|$, then the knot groups of $K_1$ and $K_2$ are profinitely equivalent, i.e. they share the same collection of finite quotient groups. It is a classical result \cite{DFPR1982} that in this case they have isomorphic profinite completions $\widehat{\pi_1}(S^3-K_1)\cong\widehat{\pi_1}(S^3-K_2)$. Therefore, it is an interesting problem to develop a theory of profinite completion for fundamental GL-racks analogous to that for knot groups. Second, since the isomorphism of the fundamental GL-racks of two Legendrian knots cannot rule out the case that both the Thurston–Bennequin number and the rotation number have opposite signs, it is necessary to investigate the coloring invariants of these two Legendrian knots under this scenario. This is precisely the topic we intend to discuss in the next chapter.
\end{enumerate}
\end{remark}

\section{Some further discussions}\label{section5}
In this section, we investigate the relationship between the coloring numbers of two Legendrian knots of the same topological knot type, under the condition that their Thurston–Bennequin numbers and rotation numbers are pairwise opposite.

\begin{theorem}
Let $K_1$ and $K_2$ be two Legendrian knots of the same knot type. If $tb(K_1)=t=-tb(K_2)$ and $rot(K_1)=r=-rot(K_2)$, then for any permutation GL-rack $(X, \ast, u, d)$, we have $\operatorname{Col}_X(K_1)=\operatorname{Col}_X(K_2)$.
\end{theorem}
\begin{proof}
Suppose that the fundamental GL-racks of $K_1$ and $K_2$ are given by
\[
\operatorname{GLR}(K_1)=\langle x_1,\dots,x_n \mid r_1,\dots,r_n \rangle \text{ and } \operatorname{GLR}(K_2)=\langle y_1,\dots,y_m \mid \tilde{r}_1,\dots,\tilde{r}_m \rangle.
\]
Adapting the argument in the proof of Theorem \ref{theorem4.2}, we obtain
\[
f(x_1) = u^{-\frac{1}{2}(q-p)} d^{\frac{1}{2}(q-p)} u^{\frac{1}{2}(p+q)} d^{\frac{1}{2}(p+q)} \sigma^\omega f(x_1) = u^{-r} d^{r} \sigma^{t} f(x_1)=u^{-r-t}d^{r-t}f(x_1)
\]
and
\[
g(y_1) = u^{-\frac{1}{2}(\tilde{q}-\tilde{p})} d^{\frac{1}{2}(\tilde{q}-\tilde{p})} u^{\frac{1}{2}(\tilde{p}+\tilde{q})} d^{\frac{1}{2}(\tilde{p}+\tilde{q})} \sigma^{\tilde{\omega}} g(y_1) = u^{r} d^{-r} \sigma^{-t} g(y_1)=u^{r+t}d^{t-r}g(y_1).
\]
The first equation holds if and only if the second one also holds, therefore the number of colorings of $K_1$ and $K_2$ are equal.
\end{proof}

In what follows, we restrict ourselves to the case of GL-quandles. While a full characterization remains elusive for two Legendrian knots $K_1$ and $K_2$ of the same knot type with opposite Thurston–Bennequin numbers and rotation numbers, we nevertheless obtain some partial results in this setting. For colorings with respect to finite GL-quandles, if $S(K)$ is obtained from $K$ by the same number of positive and negative stabilizations, due to the fact $ud=\operatorname{id}$, $K$ and $S(K)$ have the same coloring invariant. Without loss of generality, let us assume that
\begin{center}
    $tb(K_1)=t=-tb(K_2)$ and $rot(K_1)=r=-rot(K_2)\geq0$.
\end{center}
If we perform $2r$ negative stabilizations on $K_1$, its Thurston-Bennequin number becomes $t-2r$ and its rotation number becomes $-r$. At this point, the difference of the Thurston-Bennequin numbers between $K_1$ and $K_2$ is $2t-2r$, while their rotation numbers coincide. It suffices to perform $|t-r|$ positive stabilizations and $|t-r|$ negative stabilizations on either $K_2$ or $S_{-}^{2r}(K_1)$ to make their classical invariants identical. Therefore, based on the previous discussion, the problem of whether the coloring numbers of $K_1$ and $K_2$ are equal reduces to discussing whether the coloring numbers of $K_1$ and $S_{-}^{2r}(K_1)$ coincide. For the special case of $r=0$, this condition holds automatically.

Since we are only concerned with GL-quandles here, another interesting question is how the coloring invariant of a Legendrian knot $K$ is related to that of its underlying topological knot $T(K)$. We end this paper with a result in this direction.

\begin{theorem}
Suppose that the Thurston-Bennequin number and rotation number of a Legendrian knot $K$ are $t$ and $r$, respectively. Let $T(K)$ denote the underlying topological knot type of $K$, and let $(X, \ast, u, u^{-1})$ be a finite GL-quandle. Then
\[
\operatorname{Col}_X\bigl(S_{-\operatorname{sgn}(r)}^{|r|}(K)\bigr)=\operatorname{Col}_X\bigl(T(K)\bigr).
\]
\end{theorem}
\begin{proof}
With a given front projection of $K$, we can obtain a diagram of the underlying topological knot $T(K)$ by smoothing all cusps. Then we have 
    \[
    \operatorname{GLR}_{K} = \langle x_1, \cdots, x_n \mid r_1, \cdots, r_n \rangle
    \]
   and
   \[
   \operatorname{Q}(T(K)) = \langle x_1, \cdots, x_n \mid r_1', \cdots, r_n' \rangle
   \]
   Here, the relations $r_i$ are of the form
   \[
   u^{a_i}(x_i) \ast^{\varepsilon_i} x_{k_i} = x_{i+1},
   \]
   and the relations $r_i'$ are of the form
   \[
   x_i \ast^{\varepsilon_i} x_{k_i} = x_{i+1},
   \]
   where $\varepsilon_i = \pm 1$, $x_{n+1} = x_1$, and $\sum_{i=1}^n a_i = -2r$. 
   
We only consider the case $r>0$, the case $r\leq0$ is analogous. Since the stabilization of a Legendrian knot is independent of the choice of the place where the stabilization is taken, we may conclude the structure of $\operatorname{GLR}(S_{-}^r(K))$ as following.
   \[\begin{cases}
   r_1'': & u^{2r + a_1}(x_1) \ast^{\varepsilon_1} x_{k_1} = x_2, \\
   r_2'': & u^{a_2}(x_2) \ast^{\varepsilon_2} x_{k_2} = x_3, \\
   & \vdots \\
   r_{n-1}'': & u^{a_{n-1}}(x_{n-1}) \ast^{\varepsilon_{n-1}} x_{k_{n-1}} = x_n,\\
   r_n'': & u^{a_n}(x_n) \ast^{\varepsilon_n} x_{k_n} = x_1,
   \end{cases}\]
   These relations are equivalent to
   \[\begin{cases}
   r_1'': & x_1 \ast ^{\varepsilon_1} x_{k_1} = u^{-2r-a_1}(x_2), \\
   r_2'': & u^{-2r-a_1}(x_2) \ast ^{\varepsilon_2} x_{k_2} = u^{-2r-a_1-a_2}(x_3), \\
   & \vdots \\
   r_{n-1}'': & u^{-2r-\sum_{i=1}^{n-2}a_i}(x_{n-1}) \ast^{\varepsilon_{n-1}} x_{k_{n-1}} = u^{-2r-\sum_{i=1}^{n-1}a_i}(x_n),\\
   r_n'': & u^{-2r-\sum_{i=1}^{n-1}a_i}(x_n) \ast^{\varepsilon_n} x_{k_n} = u^{-2r-\sum_{i=1}^na_i}(x_1)=u^{-2r-(-2r)}(x_1)=x_1.
   \end{cases}\]
   We set
   \[
   y_1 = x_1,\quad y_2 = u^{-2r - a_1}(x_2), \cdots, y_{n-1} = u^{-2r - \sum_{i=1}^{n-2} a_i}(x_{n-1}),\quad y_n = u^{-2r - \sum_{i=1}^{n-1} a_i}(x_n).
   \]
   It then follows that $y_i$, $y_{i+1}$, and $y_{k_i}$ satisfy the relation $r_i'$. This builds a one-to-one correspondence between the colorings of $K$ and that of $T(K)$, which implies that $\operatorname{Col}_X(S_{-}^{r}(K)) = \operatorname{Col}_X(T(K))$.
\end{proof}

\begin{corollary}
All Legendrian knots of the same knot type with rotation number zero have the same coloring number with respect to any finite GL-quandle.
\end{corollary}
\begin{proof}
The coloring invariant of any Legendrian knot with rotation number zero is equal to the coloring invariant of the underlying topological knot.
\end{proof}

\bibliographystyle{plain}
\bibliography{ref}
\end{document}